\documentclass[11pt]{article}
\usepackage{tocbibind} % adds References to the pdf bookmarks, in amsart this seems to be standard, but not in article
\usepackage{amsmath,amsthm,amssymb}
\usepackage{fancyhdr}
\usepackage[british]{babel}
\usepackage{geometry,mathtools}
\usepackage{enumitem}
\usepackage{algpseudocode}
\usepackage{dsfont}
\usepackage{centernot}
\usepackage{xstring}
\usepackage{colortbl}

\usepackage[section]{algorithm}
\usepackage{graphicx}
\usepackage[font={footnotesize}]{caption}
\usepackage[usenames,dvipsnames,table]{xcolor}
\usepackage{pstricks,tikz}
\usetikzlibrary{patterns,calc,fadings,decorations.pathreplacing}
\usepackage[h]{esvect}
\usepackage[
bookmarksopen=true,
bookmarksopenlevel=1,
colorlinks=true,
linkcolor=darkblue,
linktoc=page,
citecolor=darkblue,
urlcolor=darkblue
]{hyperref}
\usepackage{bbm}

\numberwithin{equation}{section}

\definecolor{darkblue}{rgb}{0,0,0.5}

\newdimen\margin
\def\textno#1&#2\par{
   \margin=\hsize
   \advance\margin by -4\parindent
          \setbox1=\hbox{\sl#1}
   \ifdim\wd1 < \margin
      $$\box1\eqno#2$$
   \else
      \bigbreak
      \hbox to \hsize{\indent$\vcenter{\advance\hsize by -3\parindent
      \it\noindent#1}\hfil#2$}
      \bigbreak
   \fi}

\newtheorem{theorem}[algorithm]{Theorem}

\newtheorem{lemma}[algorithm]{Lemma}
\newtheorem{cor}[algorithm]{Corollary}

\theoremstyle{definition}

\def\lateproof#1{\removelastskip\penalty55\medskip\noindent\begin{stepenv}\end{stepenv}{\bf Proof of #1. }} % in each main proof, claim and step counter set back

\def\noproof{{\unskip\nobreak\hfill\penalty50\hskip2em\hbox{}\nobreak\hfill%
       $\square$\parfillskip=0pt\finalhyphendemerits=0\par}\goodbreak}
\def\endproof{\noproof\bigskip}

\def\claimproof{\removelastskip\penalty55\medskip\noindent{\em Proof of claim: }}
\def\noclaimproof{{\unskip\nobreak\hfill\penalty50\hskip2em\hbox{}\nobreak\hfill%
       $\square$\parfillskip=0pt\finalhyphendemerits=0\par}\goodbreak}

\def\endclaimproof{\noclaimproof\medskip}

\newcounter{stepenv}
\newenvironment{stepenv}[1][]{\refstepcounter{stepenv}}{}

\newcounter{step}[stepenv]

\newcounter{substep}[step]
\renewcommand{\thesubstep}{\thestep.\arabic{substep}}

\newcounter{claim}[stepenv]
\newenvironment{claim}[1][]{\refstepcounter{claim}\par\medskip\noindent%
        \textit{Claim~\theclaim. #1} \itshape\rmfamily}{\medskip}

\newcommand{\cF}{\mathcal{F}}

\newcommand{\cH}{\mathcal{H}}
\newcommand{\cI}{\mathcal{I}}

\newcommand{\cP}{\mathcal{P}}

\newcommand{\cX}{\mathcal{X}}

\newcommand{\cZ}{\mathcal{Z}}

\newcommand{\bG}{\mathbb{G}}

\newcommand{\bN}{\mathbb{N}}

\def\eps{{\epsilon}}

\newcommand{\PP}{\mathbb{P}}
\newcommand{\EE}{\mathbb{E}}

\newcommand{\defn}{\emph}

\def\lflr{\left\lfloor}
\def\rflr{\right\rfloor}
\def\lcl{\left\lceil}
\def\rcl{\right\rceil}

\newcommand{\floor}[1]{\lflr#1\rflr}
\newcommand{\ceil}[1]{\lcl#1\rcl}

\def\sm{\setminus}

\newcommand{\set}[2]{\{#1\,:\;#2\}}
\def\In{\subseteq}

\renewcommand{\eps}{\varepsilon}

\def\COMMENT#1{}
\def\TASK#1{}
\let\TASK=\footnote             % COMMENT OUT for clean output
\let\COMMENT=\footnote          % COMMENT OUT for clean output

\begin{document}

\title{Defect and transference versions of the Alon--Frankl--Lov\'asz theorem}

\author{Lior Gishboliner\thanks{Department of Mathematics, University of Toronto, Canada.
\emph{Email}: \href{mailto:lior.gishboliner@utoronto.ca}{\tt lior.gishboliner@utoronto.ca}.
}
\and 
Stefan Glock\thanks{Fakultät für Informatik und Mathematik, Universität Passau, Germany.
\emph{Email}: \href{mailto:stefan.glock@uni-passau.de}{\tt stefan.glock@uni-passau.de}, \href{mailto:amedeo.sgueglia@uni-passau.de}{\tt amedeo.sgueglia@uni-passau.de}.
SG is funded by the Deutsche Forschungsgemeinschaft (DFG, German Research Foundation) --- 542321564,
AS is funded by the Alexander von Humboldt Foundation.}
\and 
Peleg Michaeli\thanks{
    Mathematical Institute,
    University of Oxford,
    UK.
    \emph{Email}: \href{mailto:peleg.michaeli@maths.ox.ac.uk}
                {\tt peleg.michaeli@maths.ox.ac.uk}.
    Research supported by ERC Advanced Grant 883810.
    For the purpose of Open Access, the author has applied a CC BY public
    copyright licence to any Author Accepted Manuscript version arising from
    this submission.
}
\and
Amedeo Sgueglia\footnotemark[2]
}

\date{}

\maketitle

\begin{abstract} 
Confirming a conjecture of Erd\H{o}s on the chromatic number of Kneser hypergraphs, Alon, Frankl and Lov\'asz proved that in any $q$-colouring of the edges of the complete $r$-uniform hypergraph, there exists a monochromatic matching of size $\lfloor \frac{n+q-1}{r+q-1}\rfloor$. In this paper, we prove a transference version of this theorem. More precisely, for fixed $q$ and $r$, we show that with high probability, a monochromatic matching of approximately the same size exists in any $q$-colouring of a random hypergraph, already when the average degree is a sufficiently large constant.
In fact, our main new result is a defect version of the Alon--Frankl--Lov\'asz theorem for almost complete hypergraphs. From this, the transference version is obtained via a variant of the weak hypergraph regularity lemma. The proof of the defect version uses tools from extremal set theory developed in the study of the Erd\H{o}s matching conjecture.
\end{abstract}

\section{Introduction}\label{sec:intro}

A flourishing trend in combinatorics has been showing that classical theorems concerning dense graphs (or hypergraphs) have corresponding analogues in the setting of (sparse) random graphs.
Such results are usually known as \emph{transference} theorems and include, among many others, the breakthroughs of R\"odl and Ruci\'nski~\cite{RR:95} on the transference of Ramsey's theorem, and of Conlon and Gowers~\cite{CG:2016} and Schacht~\cite{schacht:2016} on the transference of Tur\'an's theorem.
Moreover, the study of combinatorial theorems for random graphs has generated several exciting recent developments in probabilistic and extremal combinatorics, including the sparse regularity method, hypergraph containers, the K{\L}R conjecture and the absorption method. 
We refer to the surveys of Conlon~\cite{conlon:2014} and B\"ottcher~\cite{boettcher:2017} (and the references therein) for more details.

Here we are interested in a transference version of the celebrated result of Alon, Frankl and Lov\'asz~\cite{AFL:86} concerning the chromatic number of Kneser hypergraphs.
In 1955, Kneser~\cite{Kneser:55} conjectured that if the $r$-subsets of a $(2r+q-1)$-element set are split into $q$ classes, then one of the classes will contain two disjoint $r$-subsets.
The conjecture remained open for 23 years, until Lov\'asz~\cite{lovasz:78} gave a topological proof using the Borsuk--Ulam theorem.
His contribution is often considered to be the start of the field of topological combinatorics, and we refer to the book of Matou{\v s}ek~\cite{Ma:03} for more examples.
More generally, given $n,r,k \in \mathbb{N}$ with $n \ge kr$, the \defn{Kneser hypergraph $KG^k(n,r)$} is the $k$-uniform hypergraph (short: \defn{$k$-graph}) where the vertices are all the $r$-subsets of $[n]$ and a collection of $k$ vertices forms an edge if the corresponding $r$-sets are pairwise disjoint.
Then, Kneser's conjecture is equivalent to the statement that $KG^2(n,r)$ is not $q$-colourable if $n \ge 2r+q-1$.
As a generalization of this, Erd\H{o}s~\cite{erdos:73} conjectured even further that $KG^k(n,r)$ is not $q$-colourable if $n \ge (q-1)(k-1)+kr$.
The case $k=2$ corresponds to Kneser's original conjecture. Moreover, the validity of the case $r=2$ is a classical result by Cockayne and Lorimer~\cite{CL:75}.
The conjecture of Erd\H{o}s was finally resolved by Alon, Frankl and Lov\'asz~\cite{AFL:86}, also using topological methods.
We state their result in an equivalent form concerning the size of a monochromatic matching that we can guarantee in any colouring of the edges of a complete hypergraph.
\begin{theorem}[Alon, Frankl and Lov\'asz~\cite{AFL:86}]
\label{thm:afl}
    Let $n,r,q \in \mathbb{N}$ with $r,q \ge 2$.
    Then any $q$-colouring of the edges of the complete $n$-vertex $r$-graph $K_n^{(r)}$ contains a monochromatic matching of size at least $\lfloor \frac{n+q-1}{r+q-1}\rfloor$. 
\end{theorem}

The bound $k:=\lfloor \frac{n+q-1}{r+q-1}\rfloor$ on the size of the matching is best possible as shown by the following construction.
Partition $[n]$ into $q$ sets $V_1,\dots,V_q$ such that $V_i$ has size at most $k$ for each $1\le i \le q-1$ and $V_q$ has size at most $rk+r-1$. 
%This is possible since $(q-1)k+rk+r-1 > (r+q-1)(\frac{n+q-1}{r+q-1}-1)+r-1 = n+q-1-(r+q-1)+r-1=n-1$.
Given an edge $e \in K_n^{(r)}$, for $1 \le i \le q-1$, we assign colour $i$ to $e$ if and only if $e$ intersects $V_i$, and we assign colour $q$ to $e$ if and only if $e$ is completely contained in $V_q$. Note that an edge might be assigned several colours $1 \leq i \leq q-1$, and in this case we choose one of these colours arbitrarily.
Then this yields an extremal $q$-colouring with respect to Theorem~\ref{thm:afl}.
Indeed, for $1 \le i \le q-1$, every edge with colour $i$ has to intersect~$V_i$ and hence any matching of colour $i$ has size at most $|V_i|\le k$. Moreover, every edge with colour $q$ is completely contained in $V_q$ and hence the size of any matching of colour $q$ is at most $\lfloor |V_q|/r\rfloor  \le k$.

Here, we prove a transference version of Theorem~\ref{thm:afl}.
The random hypergraph model we consider is the binomial random $r$-graph $\bG^{(r)}(n,p)$, which has $n$ vertices, and where each $r$-set of vertices forms an edge independently with probability~$p$.
We show that if $p \gg n^{-r+1}$, then $\bG^{(r)}(n,p)$ typically contains a monochromatic matching of asymptotically the same size as what is guaranteed to exist in $K_n^{(r)}$ by Theorem~\ref{thm:afl}.

\begin{theorem}[Transference version of the AFL Theorem]
\label{thm:main} 
For all $r, q \in \mathbb{N}$ with $r, q \ge 2$, and all $\mu >0$, there exists $C > 0$ such that, provided $p\ge Cn^{-r+1}$, w.h.p.\ the following holds for $G\sim \bG^{(r)}(n,p)$: For any $q$-colouring of the edges of $G$, there exists a monochromatic matching of size at least $(1-\mu)\frac{n}{r+q-1}$.
\end{theorem}

The size of the matching is asymptotically best possible since one cannot do better even in the complete $r$-graph, as explained above. 
In fact, it is also necessary that $C$ is sufficiently large given $\mu$, since $G$ must contain at most $O(\mu n)$ isolated vertices.
Indeed, by the optimality of Theorem~\ref{thm:afl}, if $G$ is an $n$-vertex $r$-graph with $n'$ isolated vertices, then there is a $q$-colouring of its edges whose largest monochromatic matching has size at most $\lfloor \frac{(n-n')+q-1}{r+q-1} \rfloor$.

The graph case $r=2$ (i.e.\ a transference version of the Cockayne--Lorimer theorem) was already proved by Gishboliner, Krivelevich and Michaeli~\cite{GKM:22ii}, with earlier results of Letzter~\cite{letzter:2016} and Dudek and Pra{\l}at~\cite{DP:2017} implying the cases $(r,q)=(2,2)$ and $(r,q)=(2,3)$, respectively.

We will prove Theorem~\ref{thm:main} by combining the sparse hypergraph regularity method with the following ``defect'' version of the Alon--Frankl--Lov\'asz theorem, which shows that, for large $n$, the conclusion of Theorem~\ref{thm:afl} approximately holds even for edge-colourings of \emph{almost} complete $r$-uniform hypergraphs.

\begin{theorem}[Defect version of the AFL Theorem]
\label{thm:deficiency} 
For all $r,q \in \mathbb{N}$ with $r,q\ge 2$, and all $\mu >0$, there exists $\eps >0$ such that the following holds for all sufficiently large $n$.  
Let $G$ be an $n$-vertex $r$-graph whose edges are $q$-coloured and assume $e(G)\ge (1-\eps)\binom{n}{r}$. Then $G$ contains a monochromatic matching of size at least $(1-\mu)\frac{n}{r+q-1}$.
\end{theorem}

The deduction of Theorem~\ref{thm:main} from Theorem~\ref{thm:deficiency} via the regularity method works in a similar way as in the graph case. However, finding a strategy to prove Theorem~\ref{thm:deficiency} presents several challenges.
First, the topological proof of the Alon--Frankl--Lov\'asz theorem does not seem to be robust enough to generalise to hypergraphs that miss a significant proportion of the edges. Moreover, the proof of the graph case ($r=2$) from~\cite{GKM:22ii} relies on a good understanding of matchings in graphs in the form of the Tutte--Berge formula, no analogue of which is available for hypergraphs. Finally, one could try to prove the existence of certain small coloured configurations (``gadgets'') that can be repeatedly removed until the remaining hypergraph has a very specific structure. For instance, in the case of two colours, in~\cite{GGS:23} it is implicitly proved that either there exist two edges of different colour that share $r-1$ vertices, or the colouring is almost monochromatic. For general $q$, one possible gadget would be a set of $r+q-1$ vertices that contains an edge of each colour. However, for $r \ge 3$, there are constructions where all colour classes are large yet there is no such gadget.
For instance, consider the complete $r$-graph on the disjoint union of $q-1$ sets $V_1, \dots, V_{q-1}$ of the same size, where each edge completely contained in some $V_i$ is coloured with colour $i$ and all other edges are coloured with colour~$q$.
Thus this approach seems infeasible as well.
We instead follow a new approach and make use of tools from extremal set theory developed in the study of the Erd\H{o}s matching conjecture (see the overview in Section~\ref{sec:defect} for more details).

Theorem~\ref{thm:main} also has implications to the discrepancy of perfect matchings in random hypergraphs.
The general question in discrepancy theory is whether, given a ground set $\Omega$, a family
$\cP \subseteq 2^\Omega$ and a positive integer $q \ge 2$, there exists a $q$-colouring of $\Omega$ such that each set in $\cP$ contains roughly the same number of elements of each colour.
In the setting of (hyper)graphs, the main interest is finding conditions on a hypergraph $G$ under which any $q$-edge-colouring of $G$ contains a particular substructure with high discrepancy, meaning that significantly more than a $1/q$-proportion of the substructure's edges are in the same colour.
The question goes back to works of Erd\H{o}s and Spencer~\cite{ES:72}, and of Erd\H{o}s, F\"uredi, Loebl and S\'os~\cite{EFLS:95}.
Recently, this has been extensively studied for minimum degree conditions forcing perfect matchings with high discrepancy in hypergraphs (see~\cite{BTZ-G:24,GGS:23,HLMPSTZ24+,LMX24+}), with the graph case having been considered earlier (see~\cite{BCJP:20,FHLT:21,GKM:22}).
H{\`a}n, Lang, Marciano, Pavez-Sign{\'e}, Sanhueza-Matamala, Treglown and Z{\'a}rate-Guer{\'e}n~\cite[Section 5.3]{HLMPSTZ24+} showed that there exists a constant $C>0$ such that if $p \ge C\sqrt{n^{2-r}}$, then w.h.p.\ in every $q$-colouring of the edges of $\bG^{(r)}(n,p)$ there is a perfect matching with high discrepancy.
However, due to the celebrated result of Johansson, Kahn, and Vu~\cite{JKV:08}, we know that a perfect matching already typically exists for $p \gg n^{-r+1} \log n$ (cf. Theorem~\ref{thm:JKV}).
A natural question, which was asked in~\cite{HLMPSTZ24+}, is to determine the correct threshold (depending on $r, q$) for the property of having a perfect matching with high discrepancy in every $q$-colouring.
Motivated by this, we prove the following result, which starts working at the threshold for the existence of a perfect matching, and even yields an asymptotically optimal bound on the discrepancy. However, note that we assume the colouring is known a priori, that is, the statement we prove to hold w.h.p.\ is for a fixed colouring.

\begin{theorem}
\label{thm:perfect_matching}
    For all $r, q \in \mathbb{N}$ with $r,q \ge 2$, and all $\mu>0$, there exists $C>0$ such that, provided $p\ge Cn^{-r+1}\log n$, the following holds.
    For any $q$-colouring of the $r$-subsets of $[n]$, the random hypergraph $G\sim \bG^{(r)}(n,p)$ contains w.h.p.\ a perfect matching with at least $(1-\mu)\frac{n}{r+q-1}$ edges of the same colour.
\end{theorem}

\bigskip
\noindent \textbf{Organisation.}
The rest of the paper is organised as follows.
The proof of the defect theorem (Theorem~\ref{thm:deficiency}) is discussed in Section~\ref{sec:defect} and the proof of the transference theorem (Theorem~\ref{thm:main}) in Section~\ref{sec:transference}.
The proof of Theorem~\ref{thm:perfect_matching} is also discussed there.
Finally, we give concluding remarks in Section~\ref{sec:remarks}.
The proof of the multicolour weak sparse hypergraph regularity lemma is added for completeness in Appendix~\ref{appendix:regularity}.
\bigskip

\noindent \textbf{Notation.}
We use standard graph theory notation.
In particular, we let $K_n^{(r)}$ denote the complete $r$-graph on $n$ vertices, and for a graph $G$ and a subset $W \subseteq V(G)$ we let $G[W]$ denote
the subgraph of $G$ induced by $W$.

We let $[n]$ denote the set $\{1, \dots, n\}$ and, given a set $X$ and an integer $i \ge 0$, we write $\binom{X}{i}$ for the collection of all subsets of $X$ of size $i$.

For $a, b, c \in (0, 1]$, we write $a \ll b \ll c$ in our statements to mean that there are increasing functions $f, g : (0, 1] \to (0, 1]$ such that whenever $a \le f(b)$ and $b \le g(c)$, then the subsequent result holds. 

We say that an event holds with high probability (w.h.p.) if the probability that it holds tends to $1$ as the number of vertices $n$ tends to infinity.

\section{Proof of the defect theorem}
\label{sec:defect}
We start by describing our strategy for the defect theorem (Theorem~\ref{thm:deficiency}). 
Let $G$ be an $n$-vertex $q$-edge-coloured $r$-graph $G$ with $e(G) \ge (1-o(1))\binom{n}{r}$, and recall that we want to show that $G$ contains a monochromatic matching of size $(1-o(1))\frac{n}{r+q-1}$.
Our first idea is to fix a large $k \in \mathbb{N}$ (which does not depend on $n$), and consider the family $\cF$ of the $k$-subsets $F$ of $V(G)$ for which $G[F]$ is a complete $r$-graph.
Observe that since $G$ is almost-complete, almost all $k$-subsets of $V(G)$ belong to $\cF$, i.e., $|\cF| \geq (1-o(1))\binom{n}{k}$.
It may seem at this stage that we have not really made any progress as, thinking of $\cF$ as a $k$-uniform hypergraph on $V(G)$, we still have that $\cF$ is only almost-complete.
However, we have gained the flexibility of choosing $k$, and our argument relies on choosing $k$ to be large enough in terms of $r,q$.

For each $F \in \cF$, we can apply the Alon--Frankl--Lov\'asz theorem as a black box to $G[F]$ and get a monochromatic matching $M_F$ of $G[F]$ of size roughly $\frac{k}{r+q-1}$.
Assign to $F$ the colour of the monochromatic matching $M_F$, giving a $q$-colouring of $\mathcal{F}$. 
Let blue be the most popular colour in this colouring of $\mathcal{F}$, and let $\cF':=\{F \in \cF : \text{$M_F$ is blue}\}$; so $|\cF'| \geq (1-o(1)) \binom{n}{k}/q$.
Our goal is now to find a certain structure in $\cF'$ which would translate to a large blue matching in the original hypergraph $G$.

A naive approach would be to look for a large blue matching in $\cF'$. More precisely, we need a matching of size 
$(1-o(1))\frac{n}{k}$ in $\cF'$ to obtain a matching of size $(1-o(1))\frac{n}{r+q-1}$ in $G$.
With this approach, the relevant question is what size of a matching is guaranteed to exist in an $n$-vertex $k$-graph with a given number of edges. 
This is a classical problem of Erd\H{o}s~\cite{erdos:65}, known as the Erd\H{o}s matching conjecture (short: EMC), asking, given $n,k,t \in \mathbb{N}$, to determine the maximum number of edges of an $n$-vertex $k$-graph which does not contain a matching of size $t+1$.
If $n < k(t + 1)$, the problem is trivial as the family $\binom{[n]}{k}$ itself does not contain a matching of size $t+1$.
When $n \ge k(t+1)$, Erd\H{o}s conjectured that the answer is $\max\{\binom{n}{k}-\binom{n-t}{k}, \binom{k(t+1)-1}{k}\}$. 
The two bounds correspond to two natural extremal constructions: a star-like construction, obtained by taking all edges intersecting $[t]$ in at least one vertex, and a clique-like construction, obtained by taking all edges completely contained in $[k(t+1)-1]$.
However, even if the conjecture was known to be true, it would not imply the desired bound on the size of the matching.
Indeed, for $t = (1-\eps) \frac{n}k$, the clique-like construction has density roughly $(1-\eps)^k$, and, for large $k$, the star-like construction has density roughly $1-e^{-(1-\eps)}$. Hence, the density of $\cF'$, which is about $1/q$, is not enough to guarantee the existence of an almost perfect matching.

Our second idea is that instead of a matching, we look in $\cF'$ for a collection of edges which are allowed to overlap (only) mildly. 
More precisely, our goal is to find $F_1,\dots,F_s \in \cF'$ with $s = (1-o(1))\frac{n}{k}$ such that the set $W$ of vertices appearing in more than one of the sets $F_1,\dots,F_s$ is small.
Perhaps surprisingly, such $F_1,\dots,F_s$ always exist provided that $k$ is large enough as a function of $\binom{n}{k}/|\cF'|$. 
Then, removing all edges in $\bigcup_{i=1}^s M_{F_i}$ which intersect $W$, we obtain a blue matching in $G$ of the desired size $(1-o(1))\frac{n}{r+q-1}$. 
\medskip

The key step is hence to find the desired sets $F_1,\dots,F_s \in \cF'$.
We now state the result which allows one to do this.
Roughly speaking, it says that in a hypergraph with large uniformity a constant density suffices to guarantee an ``almost-perfect almost-cover''.
\begin{theorem}\label{thm:approxmatch}
Let $1/n\ll 1/k \ll 1/C \ll \alpha,\beta$ and $\cF\subseteq \binom{[n]}{k}$ with $|\cF|\ge \beta \binom{n}{k}$. Then there exist $s:=\lceil (1-\alpha)\frac{n}{k}\rceil$ sets $F_1,\dots,F_s\in \cF$ with $|F_1\cup F_2 \cup \dots \cup F_s|\ge (k-C)s$.
\end{theorem}
We remark that, for the proof of Theorem~\ref{thm:approxmatch} to work, we only need that $(1-\alpha)^{C+1} < \beta/2$, $k \ge C+1$ and
$n \ge \max\{2k(k+1),2k/(1-\alpha)\}$.
The proof of Theorem~\ref{thm:approxmatch} uses tools from extremal set theory and we start by introducing the relevant definitions.
For a $k$-graph $\cF\subseteq\binom{[n]}{k}$
and $0\le \ell<k$,
we let $\sigma_\ell(\cF)$ denote the \emph{$\ell$-shadow} of $\cF$.
Namely,
\[
  \sigma_\ell(\cF) := \left\{
    G\in\binom{[n]}{\ell}:
    \exists F\in\cF \text{ with } G\subseteq F
  \right\}\, ,
\]
where we observe in particular that if $\cF \neq \emptyset$, then $\sigma_0(\cF)=\{\emptyset\}$ and $|\sigma_0(\cF)|=1$.
In his work on the Erd\H{o}s matching conjecture, Frankl~\cite{frankl:13} proved the following fundamental inequality which generalises Katona's intersection theorem~\cite{katona:64}: For any integer $s \ge 2$ and any $\cF\In \binom{[n]}{k}$ which does not contain $s$ pairwise disjoint sets, it holds that 
\begin{align}
|\sigma_{k-1}(\cF)|\ge |\cF|/(s-1). \label{Frankl shadow}
\end{align}
Note that $\cF$ is a family as above if and only if $|F_1 \cup F_2 \cup \dots \cup F_s| \le ks -1$ for each $F_1, F_2, \dots, F_s \in \cF$, as otherwise the sets $\{F_i:i \in [s]\}$ would be pairwise disjoint.
We can generalise~\eqref{Frankl shadow} as follows.

\begin{theorem}\label{thm:shadow}
Let $k \ge 1$, $s\ge 2$ and $1 \le b \le k$. Assume that $\cF\In \binom{[n]}{k}$ is a family such that 
\begin{equation}
\label{eq:shadow_k-b_hypotheses}
|F_1\cup F_2 \cup \dots \cup F_s|\le ks - (b-1)(s-1) -1
\end{equation}
for all $F_1,\dots,F_s \in \cF$. Then 
\begin{equation}
\label{eq:shadow_k-b}
|\sigma_{k-b}(\cF)|\ge \frac{|\cF|}{(s-1)^b}\, .
\end{equation}
\end{theorem}

%\proof
%We want to apply Theorem~\ref{thm:frankl} with $t:=(b-1)(s-1)+1$.
%Define $f(x):=(x+t)/(s-1)$ and observe that $b=\lceil t/(s-1)\rceil$. Moreover, for each $0 \le i,f(i) \le k$, we have
%\begin{align*}
%\binom{i+\lceil f(i)\rceil}{\lceil f(i)\rceil-b} \bigg/ \binom{i+\lceil f(i)\rceil}{\lceil f(i)\rceil} &= \frac{\lceil f(i)\rceil \cdot (\lceil f(i)\rceil-1)\cdot \, \cdots \, \cdot (\lceil f(i)\rceil-b+1) }{(i+b)\cdot (i+b-1) \cdot \, \cdots \, \cdot (i+1)} \\
%&\ge \frac{\frac{i+b}{s-1}\cdot \frac{i+b-1}{s-1} \cdot \, \cdots \, \cdot \frac{i+1}{s-1}}{(i+b) \cdot (i+b-1) \cdot \, \cdots \, \cdot (i+1)} \\
%&= \frac{1}{(s-1)^b},
%\end{align*}
%where we used that $\lceil f(i)\rceil \ge f(i) = \frac{i+(b-1)(s-1)+1}{s-1} \ge \frac{i+b-j}{s-1}+j$ for any $j\in \Set{0,\dots,b-1}$.
%Hence, $\alpha(f,k,b)\ge \frac{1}{(s-1)^b}$, as desired.
%\endproof

We remark that Theorem~\ref{thm:shadow} can be deduced from a technical result of Frankl (c.f.~\cite[Corollary~4]{frankl:91}).
Here, for the sake of completeness, we give a direct proof, which proceeds analogously to Frankl's proof~\cite{frankl:13} of~\eqref{Frankl shadow}.

\lateproof{Theorem~\ref{thm:shadow}}
    It is well known (see, for example,~\cite[Proposition~$1$]{frankl:91}) that we can assume that $\cF$ is a \defn{shifted} family, that is, for all $1 \le i < j \le  n$ and $F \in \cF$, the conditions $i \not\in F$, $j \in F$ imply that $F \cup \{i\} \setminus \{j\} \in \cF$ as well.

    We first prove the statement for $k=1$ (and thus $b=1$) and any $s \ge 2$.
    Observe that if $F_1, \dots, F_s \in \binom{[n]}{1}$ and $|F_1 \cup F_2 \cup \dots \cup F_s| < s$, then the $F_i$ cannot be all distinct.
    Therefore if $\cF \subseteq \binom{[n]}{1}$ satisfies~\eqref{eq:shadow_k-b_hypotheses}, then $|\cF| \le s-1$.
    Since $|\sigma_{0}(\cF)|=1$,~\eqref{eq:shadow_k-b} follows.

    Assume now that $k \ge 2$.
    We first prove the statement for all $n \le ks-(b-1)(s-1)-1$.
    Consider the bipartite graph with partite sets $\cF$ and $\sigma_{k-b}(\cF)$, where there is an edge connecting $F \in \cF$ and $G \in \sigma_{k-b}(\cF)$ if and only if $G \subseteq F$. 
    Each $F \in \cF$ has degree $\binom{k}{k-b}=\binom{k}{b}$ and each $G \in \sigma_{k-b}(\cF)$ has degree at most $\binom{n -|G|}{b} = \binom{n - k + b}{b}$. 
    Therefore, by double counting the edges, $\binom{k}{b} \cdot |\cF| \le \binom{n - k + b}{b} \cdot |\sigma_{k-b}(\cF)|$, implying
    \[
        \frac{|\cF|}{|\sigma_{k-b}(\cF)|} \le \frac{\binom{n - k + b}{b}}{\binom{k}{b}} = \prod_{j=0}^{b-1} \frac{n-k+b-j}{k-j} \le (s-1)^b\, ,
    \]
    where the last inequality can be justified as follows.
    Since $n \le ks-(b-1)(s-1)-1$, it is enough to show that $\frac{ks-(b-1)(s-1)-1-k+b-j}{k-j} \le s-1$, which is equivalent to $(b-j-1)(s-2) \ge 0$.
    This is clearly true as $s \ge 2$ and $0 \le j \le b-1$.
    We obtain~\eqref{eq:shadow_k-b} by rearranging.

    We now suppose that $n \ge ks-(b-1)(s-1)$ and proceed by induction on $n$. So assume that the statement holds for $n - 1$ for both $k$ and $k-1$.
    Define 
    \begin{equation*}
        \cF(\overline{n}) := \{F \in \cF: n \not\in F\} \quad \text{ and } \quad \cF(n) := \{F \setminus \{n\}: F \in \cF, n \in F\}\, ,
    \end{equation*}
    and observe that $\cF(\overline{n}) \subseteq \binom{[n-1]}{k}$ and $\cF(n) \subseteq \binom{[n-1]}{k-1}$.
    Since $\cF(\overline{n}) \subseteq \cF$, for any $F_1, \dots, F_s \in \cF(\overline{n})$ we have 
    $$|F_1\cup F_2 \cup \dots \cup F_s|\le ks - (b-1)(s-1) -1\, .$$
    Therefore, by induction, $|\cF(\overline{n})| \le (s-1)^b \cdot |\sigma_{k-b}(\cF(\overline{n}))|$.
    Moreover, we claim that for any $G_1, \dots, G_s \in \cF(n)$ we must have 
    \begin{equation}
    \label{eq:G}
        |G_1\cup G_2 \cup \dots \cup G_s|\le (k-1)s - (b-1)(s-1) -1\, .
    \end{equation}
    Indeed, suppose this is not the case and let $c \ge 0$ be such that 
    $$|G_1\cup G_2 \cup \dots \cup G_s| = (k-1)s - (b-1)(s-1)+ c\, .$$
    By definition of $\cF(n)$, we have $F_i:=G_i \cup \{n\} \in \cF$ and thus 
    $$ks - (b-1)(s-1) -1 \ge |F_1\cup F_2 \cup \dots \cup F_s| = |G_1\cup G_2 \cup \dots \cup G_s| +1\, ,$$
    from which it follows that $c \le s-2$.
    Moreover, by our assumption on $n$, we have
    $$n-1 \ge ks-(b-1)(s-1) -1 = |G_1\cup G_2 \cup \dots \cup G_s| +s-c-1,$$ 
    and $s-c-1 \ge 1$.
    Therefore there exist distinct $a_1, \dots, a_{s-c-1} \in [n-1]$ which do not belong to any of $G_1,\dots,G_s$. 
    Let $F_i':= F_i \cup \{a_i\} \setminus \{n\}$ for $i \le s-c-1$, and $F_i':=F_i$ for $s-c \leq i \leq s$.
    Then, since $\cF$ is shifted, $F_i' \in \cF$ for each $i \in [s]$.
    However, 
    $$|F_1' \cup F_2' \cup \dots \cup F_s'| = |G_1\cup G_2 \cup \dots \cup G_s| + s-c-1+1 = ks - (b-1)(s-1)\, ,$$
    which contradicts our assumption.
    Hence,~\eqref{eq:G} holds.
    
    If $b \le k-1$, then we can apply induction and get that $|\cF(n)| \le (s-1)^b \cdot |\sigma_{k-1-b}(\cF(n))|$.
    Note that $|\cF|=|\cF(\overline{n})|+|\cF(n)|$.
    Moreover, if $A \in \sigma_{k-b}(\cF(\overline{n}))$, then $n \not\in A$ and $A \in \sigma_{k-b}(\cF)$.
    Similarly, if $B \in \sigma_{k-1-b}(\cF(n))$, then $n \not \in B$ and $B \cup \{n\} \in \sigma_{k-b}(\cF)$.
    Therefore 
    $$|\sigma_{k-b}(\cF)| \ge |\sigma_{k-b}(\cF(\overline{n}))| + |\sigma_{k-1-b}(\cF(n))|\, ,$$
    and
    \[
        |\cF|=|\cF(\overline{n})|+|\cF(n)| \le (s-1)^b \cdot \bigg(|\sigma_{k-b}(\cF(\overline{n}))| + |\sigma_{k-1-b}(\cF(n))|\bigg) \le (s-1)^b \cdot |\sigma_{k-b}(\cF)|\, ,
    \]
    which proves~\eqref{eq:shadow_k-b} when $b \le k-1$.
    
    When $b=k$,~\eqref{eq:G} reads as $|G_1 \cup \dots \cup G_s| \le (k-1)s-(k-1)(s-1) - 1 =k-2$, which cannot be satisfied as $|G_1 \cup \dots \cup G_s| \ge |G_1|=k-1$. Therefore $\cF(n)=\emptyset$ and
    \[
        |\cF|=|\cF(\overline{n})| \le (s-1)^b \cdot |\sigma_{0}(\cF(\overline{n}))| = (s-1)^b \cdot |\sigma_{0}(\cF)|\, ,
    \]
    which proves~\eqref{eq:shadow_k-b} in the remaining case $b=k$.
\endproof

We can now prove Theorem~\ref{thm:approxmatch}.

\lateproof{Theorem~\ref{thm:approxmatch}}
Suppose the statement is false.
Then for any $F_1,\dots,F_s\in \cF$, we have
$|F_1\cup F_2 \cup \dots \cup F_s|< (k-C)s$ and, choosing $b:=C+1$, we have $(k-C)s \le ks - (b-1)(s-1) -1$. Moreover, since $1/n \ll 1/k,\alpha$, it holds that $s \ge 2$ and thus we can use Theorem~\ref{thm:shadow} (applied with the above choice of $b$) and we have
$$|\sigma_{k-b}(\cF)|\ge \frac{|\cF|}{(s-1)^b}.$$
In particular,
$$|\sigma_{k-b}(\cF)|\ge \frac{\frac{\beta}{2}\frac{n^k}{k!}}{(1-\alpha)^b (n/k)^b} \ge \frac{\beta}{2(1-\alpha)^b} \binom{n}{k-b} > \binom{n}{k-b}.$$
where the first inequality follows from $s-1\le (1-\alpha)\frac{n}{k}$ and $|\cF|\ge \beta \binom{n}{k} \ge \frac{\beta}{2}\frac{n^k}{k!}$, the second inequality from $\frac{k^b}{k!}\ge \frac{1}{(k-b)!}$ and the last one uses that $b=C+1$ and $1/C \ll \alpha, \beta$.
This is a contradiction to the trivial upper bound $|\sigma_{k-b}(\cF)|\le \binom{n}{k-b}$.
\endproof

We conclude this section by proving Theorem~\ref{thm:deficiency}, the defect version of the Alon--Frankl--Lov\'asz theorem. 

\lateproof{Theorem~\ref{thm:deficiency}}
First observe that we can assume that $\mu \ll 1/r,1/q$.
Then choose new constants $k,C$ such that $k \in \mathbb{N}$ and $1/n\ll \eps \ll 1/k \ll 1/C \ll \mu \ll 1/r,1/q$.
Let $\cF\In \binom{V(G)}{k}$ be the family of $k$-subsets $F\In V(G)$ for which $G[F]$ is the complete $r$-graph on $F$.
\begin{claim}
    We have $|\cF|\ge \frac{1}{2}\binom{n}{k}$.
\end{claim}
\claimproof
    Let $F$ be a uniformly chosen subset of $V(G)$ of size $k$ and observe that it is enough to show that $\PP[F \not\in \cF] \le 1/2$, where we recall that $F \not\in \cF$ is equivalent to the condition that $G[F]$ contains a non-edge.    
    Given $r$ vertices $v_1, \dots, v_r$, we have $\PP[v_1, \dots, v_r \in F] = \binom{n-r}{k-r} \cdot \binom{n}{k}^{-1}$ and thus the expected number of non-edges in $G[F]$ is at most $\eps \cdot \binom{n}{r} \cdot \binom{n-r}{k-r} \cdot \binom{n}{k}^{-1} = \eps \binom{k}{r} \le 1/2$, where the inequality uses $\eps \ll 1/k, 1/r$.
    By Markov's inequality, the probability that $G[F]$ contains a non-edge is at most $1/2$, and we are done.
\endclaimproof
Using the Alon--Frankl--Lov\'asz theorem (Theorem~\ref{thm:afl}), for each $F\in \cF$, we find in $G[F]$ a monochromatic matching, which we denote by $M_F$, of size at least $\lfloor \frac{k}{r+q-1}\rfloor \ge (1-\mu^2)\frac{k}{r+q-1}$, using $1/k \ll \mu, 1/r, 1/q$. By averaging, there exists a colour $c$ and a subfamily $\cF'\In \cF$ of size at least $\frac{|\cF|}{q} \ge \frac{1}{2q}\binom{n}{k}$ such that for each $F\in \cF'$, the monochromatic matching $M_F$ has colour $c$.

Applying Theorem~\ref{thm:approxmatch} to $\cF'$ gives that there exist $F_1,\dots,F_s \in \cF'$ with $|F_1\cup F_2 \cup \dots \cup F_s|\ge (k-C)s$, where $s:=\lceil (1-\mu^2)\frac{n}{k}\rceil$.
We would like to estimate the number of edges appearing in more than one of the matchings $M_{F_i}$ and, for that, we first estimate the number of vertices appearing in more than one of the sets $F_i$.

Let $U:=F_1\cup F_2 \cup \dots \cup F_s$ and note that $|U|\ge (1-2\mu^2)n$. For a vertex $u\in U$, let $d(u):=|\set{i\in [s]}{u\in F_i}|$ and set $W:=\{u \in U: d(u) \ge 2\}$, i.e.~$W$ is the set of vertices which belong to at least two of the sets $F_i$. By double-counting, we have $$|U| + \sum_{u\in W} (d(u)-1) = \sum_{u\in U} d(u) = \sum_{i \in [s]} |F_i| = ks \le n,$$ so $\sum_{u\in W} (d(u)-1) \le n-|U| \le 2\mu^2 n$.
Since $d(u)\le 2(d(u)-1)$ for all $u\in W$, we then deduce that $\sum_{u\in W} d(u) \le 4\mu^2 n$.
Let $M$ be the multiset consisting of all the edges in $M_{F_i}$ for $i \in [s]$ and delete all the edges which contain a vertex from $W$. The remaining edges then form a matching. Note that, for each vertex $u\in W$, we delete at most $d(u)$ edges since for each $i\in [s]$, at most one edge of $M_{F_i}$ contains $u$. Hence, the final matching has size at least
$$s\cdot (1-\mu^2)\frac{k}{r+q-1} - \sum_{u\in W} d(u)  \ge (1-2\mu^2)\frac{n}{k} \cdot \frac{k}{r+q-1} - 4\mu^2 n \ge (1-\mu )\frac{n}{r+q-1}\, ,$$
and it is monochromatic (of colour $c$) by construction.
\endproof

\section{Proof of the transference and the discrepancy theorems}
\label{sec:transference}
The transference theorem (Theorem~\ref{thm:main}) follows from our defect theorem via the multicolour weak sparse hypergraph regularity lemma, which we now state after introducing a few definitions.

Let $H$ be an $r$-graph.
For disjoint vertex sets $X_1,\dots,X_r$ denote by $E(X_1,\dots,X_r)$ the set of edges having exactly one vertex in each $X_i$, $i=1,\dots,r$.
The \defn{density} between $X_1,\dots,X_r$ is
\[
  d(X_1,\dots,X_r) := \frac{|E(X_1,\dots,X_r)|}{|X_1|\cdots|X_r|}.
\]
For $\eps>0$ and $p\in[0,1]$,
an $r$-partite $r$-uniform hypergraph with parts $V_1,\dots,V_r$ is \defn{$(\eps,p)$-regular}
if for every $X_i\subseteq V_i$, $i=1,\dots,r$, with $|X_i|\ge\eps|V_i|$, one has
\[
  |d(X_1,\dots,X_r)-d(V_1,\dots,V_r)| \le \eps p.
\]
A partition $\cP:=(V_1,\dots,V_t)$ of $H$ is called \defn{$(\eps,p)$-regular}
if it is an \defn{equipartition} (i.e., the sizes of the parts differ by at most $1$)
and for all but at most $\eps\binom{t}{r}$ of the $r$-sets $(V_{i_1},\dots,V_{i_r})$,
the induced $r$-partite $r$-uniform hypergraph with parts $V_{i_1},\dots,V_{i_r}$ is $(\eps,p)$-regular.
Moreover we say that $t$ is the \emph{order} of $\cP$.
For $\eta>0$ and $D>1$, we say that $H$ is \defn{$(\eta,p,D)$-upper-uniform}
if for any disjoint vertex sets $X_1,\dots,X_r$ with $|X_i|\ge\eta|V(H)|$
one has $d(X_1,\dots,X_r)\le Dp$.

We can now state the lemma we need. 
The proof is a straightforward adaptation of a proof of the sparse regularity lemma for graphs (see, for example, the survey of Gerke and Steger~\cite{SS_survey:05}).
We provide its proof for completeness in Appendix~\ref{appendix:regularity}.
\begin{lemma}\label{lem:reg:clr}
  Let $\eta \ll 1/T \ll 1/t_0 \ll \eps, 1/D, 1/q, 1/r$, with $T,D,q,t_0,r \in \mathbb{N}$ and $D > 1$.
  Then for every $p\in[0,1]$,
  if $H_1,\dots,H_q$ are $(\eta,p,D)$-upper-uniform $r$-graphs on the same vertex set $V$
  with $|V|\ge t_0$
  then there exists a partition $V_1,\dots,V_t$ of $V$
  with $t_0\le t\le T$
  which is $(\eps,p)$-regular with respect to $H_j$ for all $j=1,\dots,q$.
\end{lemma}

Given an $(\eps,p)$-regular partition $V_1,\dots,V_t$ of $V$ as in Lemma~\ref{lem:reg:clr}, we define the \emph{cluster hypergraph} with respect to $\{V_1,\dots,V_t\}$ as the $r$-graph on $[t]$ with an edge $e:=\{i_1, i_2, \dots, i_r\}$ if and only if there exists $i:=i(e) \in [q]$ such that the induced $r$-partite $r$-subgraph of $H_i$ with parts $V_{i_1}, \dots, V_{i_r}$ is $(\eps,p)$-regular and $d_{H_i}(V_{i_1}, \dots, V_{i_r}) > \eps p$.
The cluster hypergraph inherits a $q$-edge-colouring from $H$ by colouring the edge $e$ in colour $i(e)$ (if there is more than one choice for $i(e)$, we pick one arbitrarily).

We apply Lemma~\ref{lem:reg:clr} to $\bG^{(r)}(n,p)$ and thus we need to show that $\bG^{(r)}(n,p)$ is upper-uniform.

\begin{lemma}
\label{lem:G(n,r,p)}
    Let $1/C \ll \eta, 1/r$ with $r \in \mathbb{N}$ and $r \ge 2$.
    Provided that $p \ge Cn^{-r+1}$, the random hypergraph $G \sim \bG^{(r)}(n, p)$ has w.h.p.\ the following property: for any $r$ pairwise disjoint vertex sets $X_1, \dots, X_r$, each having size at least $\eta n$, it holds that $p/2 \le d(X_1,\dots,X_r) \le 3p/2$.
\end{lemma}

Lemma~\ref{lem:G(n,r,p)} follows from a simple application of the following version of Chernoff's bound (see e.g.~\cite[Corollary~2.3]{JLR:00}).

\begin{lemma}[Chernoff's inequality]
\label{lem:chernoff}
     Let $X$ be the sum of independent Bernoulli random variables. Then for all $0 < \beta < 1$, we have
     \[
     \PP\Big[|X - \EE[X]| \ge \beta \EE[X]\Big] \le 2\exp\left(-\frac{\beta^2}{3} \EE[X]\right)\, .
     \]
\end{lemma}

\lateproof{Lemma~\ref{lem:G(n,r,p)}}
    Consider any $r$ pairwise disjoint vertex sets $X_1, \dots, X_r$ each having size at least $\eta n$. Then $\EE[e(X_1, \dots, X_r)]=p|X_1|\cdot \dots \cdot |X_r| \ge (\eta n)^r p$.
    Therefore, by Chernoff's bound, we have
    \[
        \PP\Bigg[\Big|e(X_1,\dots,X_r) -p|X_1|\cdot \dots \cdot |X_r|\Big|\ge \frac{1}{2} \cdot p |X_1|\cdot \dots \cdot |X_r|\Bigg] \le 2 \exp\left(-\frac{(\eta n)^r}{12} p\right) \le 2 \exp\left(-\frac{C \eta^r}{12} n\right)
    \]
    which, by choosing $C$ large enough, beats the union bound over the at most $2^{rn}$ choices of $X_1, \dots, X_r$.
\endproof

After applying the regularity lemma to $G \sim \bG^{(r)}(n,p)$, we consider the corresponding reduced hypergraph $R$. 
Such a graph is almost complete and inherits an edge-colouring from $G$, by colouring an edge of $R$ in colour $c$ if the $r$-partite subgraph of $G$ induced by the corresponding clusters is dense in colour $c$. In particular, $R$ is suitable for an application of Theorem~\ref{thm:deficiency} and contains a
monochromatic matching of size roughly $\frac{|V(R)|}{r+q-1}$. 
By applying the following standard technique, this monochromatic matching translates into a monochromatic matching of $G$ of size roughly $\frac{|V(G)|}{r+q-1}$, as desired.

\begin{lemma}
\label{lem:from_reduced}
    Let $\eps>0$, $p \in [0,1]$ and $G$ be an $(\eps,p)$-regular $r$-partite $r$-graph with parts $V_1, \dots, V_r$, each of size $m$ or $m+1$, and with $d(V_1,\dots,V_r) > \eps p$. Then $G$ has a matching of size at least $(1 - \eps)m$.
\end{lemma}
\begin{proof}
    Let $M$ be a maximum matching of $G$ and suppose that $|M| < (1-\eps)m$.
    Let $X_i:=V_i \setminus V(M)$ be the set of uncovered vertices in $V_i$ and observe that $|X_i| \ge \eps m$ for each $i \in [r]$.
    Then, by the definition of $(\eps,p)$-regularity, $d(X_1,\dots,X_r) \ge d(V_1,\dots,V_r) - \eps p >0$ and thus the sets $X_1,\dots,X_r$ span at least one edge, contradicting the maximality of~$M$.
\end{proof}

We are now ready to prove Theorem~\ref{thm:main}.
\lateproof{Theorem~\ref{thm:main}}
    Let $\eps, \eta>0$ and $t_0,T \in \mathbb{N}$ be new constants such that
    \[
        1/C \ll \eta \ll 1/T \ll 1/t_0 \ll \eps \ll \mu, 1/r, 1/q.
    \]
    Then let $p \ge C n^{-r+1}$ and let $G \sim \bG^{(r)}(n,p)$.
    By Lemma~\ref{lem:G(n,r,p)}, w.h.p.\ $G$ has the property that for any $r$ pairwise disjoint vertex sets $X_1, \dots, X_r$, each having size at least $\eta n$, it holds that
    \begin{equation}
    \label{eq:density}
        p/2 \le d_G(X_1,\dots,X_r) \le 3p/2\, .
    \end{equation}
    Fix now such a $G$ and consider a $q$-colouring of the edges of $G$ with colours in $[q]$. For each $i \in [q]$, let $G_i$ be the spanning subgraph of $G$ consisting of the edges of colour $i$.
    Observe that by~\eqref{eq:density}, $G$ is $(\eta,p,3/2)$-upper-uniform and thus each $G_i$ is $(\eta,p,3/2)$-upper-uniform as well.

    By the weak sparse hypergraph regularity lemma (Lemma~\ref{lem:reg:clr}), there exists a partition $V_1, \dots, V_t$ of $V(G)$ with $t_0 \le t \le T$ which is $(\eps,p)$-regular with respect to $G_i$ for each $i \in [q]$.
    Let $R$ be the ($q$-edge-coloured) cluster hypergraph associated with this partition.
    Observe that, by the definition of an $(\eps,p)$-regular partition, all but at most $q \cdot \eps \binom{t}{r}$ of the sets $\{i_1, i_2, \dots, i_r\} \subseteq [t]$ are such that the induced $r$-partite $r$-subgraph of $H_i$ with parts $V_{i_1}, \dots, V_{i_r}$ is $(\eps,p)$-regular for each $i \in [q]$.
    Moreover, by~\eqref{eq:density}, for any such $r$-set $\{i_1, i_2, \dots, i_r\}$ it holds that $d_G(V_{i_1}, \dots, V_{i_r}) \ge p/2$ and thus, by averaging, there exists $i \in [q]$ such that $d_{G_i}(V_{i_1}, \dots, V_{i_r}) \ge p/(2q) > \eps p$.
    Therefore $e(R) \ge (1-q \cdot \eps) \binom{t}{r}$.

     By Theorem~\ref{thm:deficiency} (applied with $\mu/2$ playing the role of $\mu$), $R$ contains a monochromatic matching of size at least $(1-\frac{\mu}{2}) \frac{t}{r+q-1}$, say of colour $1$.
     By Lemma~\ref{lem:from_reduced}, each edge $\{i_1,\dots,i_r\}$ of $M$ gives rise (in $G$) to a monochromatic matching (of colour $1$) of size at least $(1-\eps)\lfloor n/t \rfloor$ between the sets $V_{i_1},\dots,V_{i_r}$. Moreover, these matchings  (for different choices of $\{i_1,\dots,i_r\}$) are pairwise vertex-disjoint.
     Therefore the union of these matchings is a monochromatic matching (of colour~$1$) in $G$ of size at least $(1-\frac{\mu}{2}) \frac{t}{r+q-1} \cdot (1-\eps)\lfloor \frac{n}{t} \rfloor \ge (1-\mu) \frac{n}{r+q-1}$, as wanted.
\endproof

Finally we prove Theorem~\ref{thm:perfect_matching}, which follows by combining our Theorem~\ref{thm:main} and (a special case of) a result due to Johansson, Kahn, and Vu~\cite{JKV:08}, which we state below.

\begin{theorem}[Johansson, Kahn, and Vu~\cite{JKV:08}]
\label{thm:JKV}
    For all $r \in \mathbb{N}$ with $r \ge 2$, there exists $C >0$ such that, provided $p \ge Cn^{-r+1}\log n$, w.h.p.\ $\bG^{(r)}(n,p)$ has a perfect matching.
\end{theorem}

\lateproof{Theorem~\ref{thm:perfect_matching}}
    Let $C:=C(q,r,\mu)$ be a large enough constant.
    We expose the random hypergraph $G \sim \bG^{(r)}(n,C n^{-r+1}\log n)$ in two stages, and will after each step fix an outcome that holds with high probability.
    For that, pick $C'$ such that $(1-Cn^{-r+1})(1-C'n^{-r+1}\log n) = (1-Cn^{-r+1}\log n)$ (then $C' \ge C/2$), and let $G_1 \sim \bG^{(r)}(n,Cn^{-r+1})$ and $G_2 \sim \bG^{(r)}(n,C'n^{-r+1}\log n)$ be independent binomial random $r$-graphs on $[n]$.
    Then $G \sim G_1 \cup G_2$.
    
    Fix a $q$-colouring of the $k$-subsets of $[n]$.
    Reveal the edges of $G_1$.
    Then, by Theorem~\ref{thm:main}, w.h.p.\ $G_1$ contains a monochromatic matching of size $\ceil{(1-\mu)\frac{n}{r+q-1}}$.
    We fix such an outcome, and let $M_1$ denote this matching and $W:=[n] \setminus V(M_1)$.

    Reveal now the edges of $G_2[W]$.
    Observe $n^{-r+1}\log n = \Omega_{r,q,\mu}(|W|^{-r+1} \log |W|)$ since $|W|=\Theta_{r,q,\mu}(n)$, and thus, by Theorem~\ref{thm:JKV}, w.h.p.\ $G_2[W]$ has a perfect matching.
    We fix such an outcome and let $M_2$ denote this matching.

    This concludes the proof as $M_1 \cup M_2$ is a perfect matching of $G$ with at least $|M_1| \ge (1-\mu)\frac{n}{r+q-1}$ edges of the same colour, as desired.
\endproof

Recently Kahn~\cite{kahn:23} determined the sharp threshold for the existence of a perfect matching in hypergraphs, proving that the conclusion of Theorem~\ref{thm:JKV} holds with $C=(r-1)!+\eps$ for any $\eps>0$, and that the constant $(r-1)!$ cannot be improved.
We conjecture that Theorem~\ref{thm:perfect_matching} is true already at the sharp threshold, namely that, for any $r,q \in \mathbb{N}$ and $\eps, \mu >0$, its conclusion holds for $C= (r-1)!+\eps$.
Note that, in order to achieve that, only replacing Theorem~\ref{thm:JKV} with the sharp version is not enough: indeed, using the same notation as in the proof, since the size of $W$ is much smaller than $n$, having $p  = ((r-1)! + \varepsilon)n^{-r+1} \log n$ is not enough to guarantee a perfect matching in $G_2[W]$.

\section{Concluding remarks}
\label{sec:remarks}

In this paper, we proved a transference and a defect version of the Alon--Frankl--Lov\'asz theorem. 
It would be very interesting to characterize the extremal colourings for the AFL Theorem, that is, those $q$-colourings for which the bound in Theorem~\ref{thm:afl} is tight. In the graph case $r=2$, the extremal colourings were characterized in~\cite{XYZ:20} using the Gallai--Edmonds decomposition theorem. 
Closely related to this, it would be desirable to have a stability version, which should say that any $q$-colouring for which the largest monochromatic matching has size at most $(1+\mu)\frac{n}{r+q-1}$ must be $\eps$-close to one of the extremal examples, that is, by recolouring at most $\eps n^r$ edges, we obtain one of the extremal examples.

We point out that the colouring described after Theorem~\ref{thm:afl} is definitely not the only extremal example. For instance, let $x_1,\dots,x_q$ be any positive integers such that $x_1+\dots+x_q=r+q-1$.
(The choice $x_1=\dots=x_{q-1}=1$ and $x_q=r$ corresponds to the construction in Section~\ref{sec:intro}.)
Then partition $[n]$ into sets $V_1,\dots,V_q$ such that $|V_i|=x_i \cdot \frac{n}{r+q-1}$ (we are ignoring rounding issues here).
For any edge $e \in K_n^{(r)}$, there must be $i \in [q]$ such that $|e\cap V_i|\ge x_i$. (If not, $|e|\le \sum_{i=1}^{q} |e\cap V_i|\le \sum_{i=1}^{q}(x_i-1) = r+q-1-q<r$.)
If there are multiple such $i$, just pick one arbitrarily.
Then colour $e$ with colour~$i$. 
Observe that, for every $i \in [q]$, every edge with colour $i$ intersects $V_i$ in at least $x_i$ vertices and thus any matching in colour $i$ has size at most $|V_i|/x_i=\frac{n}{r+q-1}$.

One of the referees raised the following additional problem, which seems still open:
Given $r,q \in \bN$ with $r,q \ge 2$, what is the threshold above which w.h.p.\ every $q$-colouring of the edges of $\bG^{(r)}(n, p)$ contains a monochromatic matching of the same size as that given by the Alon--Frankl--Lov\'asz theorem?
That is, which $p$ allows one to remove the error term in Theorem~\ref{thm:main}?
We can observe the following for the graph case with two colours (corresponding to $r=q=2$) where, assuming for simplicity that $n$ is divisible by $3$, we aim for a monochromatic matching of size $n/3$.
Corollary~1.4 in~\mbox{%DIFAUXCMD
\cite{ORR:18} }\hskip0pt%DIFAUXCMD
implies that $p>2/3$ suffices.
On the other hand, we can easily observe that the threshold has to be at least constant.
Indeed, suppose that an $n$-vertex graph $G$ has two distinct non-adjacent vertices, say $v$ and $w$, of degree smaller than $n/6$.
Then partition the vertex set in two sets $A,B$ with $|A|=n/3$ such that $v$ and its neighbours are in $A$, $w \in B$ and the neighbours of $w$ are in $A$.
Colour by red all the edges in $B$ and by blue every other edge.
Then $w$ is isolated in $G[B]$ and any blue matching saturating $A$ must use an edge inside $A$ to cover $v$.
Therefore $G$ has no monochromatic matching of size $n/3$.

\section*{Acknowledgements}

We thank the anonymous referees for their valuable comments.

\bibliographystyle{amsplain_v2.0customized}
\bibliography{References}

\appendix

\section{Proof of Lemma~\ref{lem:reg:clr}}
\label{appendix:regularity}
\newcommand{\enr}{\mathcal{E}}

\textbf{Notation and definitions.} We introduce some notation.
We fix a ground set $V$ throughout the appendix.
For a family $\cX=\{X_1,\dots,X_\ell\}$ of pairwise disjoint non-empty subsets of $V$,
write $\pi(\cX):=\prod_{i=1}^\ell |X_i|$ and $\alpha(\cX):=\prod_{i=1}^{\ell} \frac{|X_i|}{|V|}=|V|^{-|\cX|} \cdot \pi(\cX)$.
Moreover, for a set $\cZ$ and an integer $r$, when we write $\binom{\cZ}{r}$ we implicitly assume that $r \le |\cZ|$.

We start from the following easy lemma.
\begin{lemma}\label{lem:alpha}
  Let $\cZ=\{Z_1,\dots,Z_\ell\}$ be a family of pairwise disjoint subsets of $V$.
  Then
  \[
    \sum_{\cX\in\binom{\cZ}{r}}\alpha(\cX)\le \left(\frac{|Z_1 \cup \dots \cup Z_\ell|}{|V|}\right)^r\, .
  \]
  %In particular, if $\cZ$ is a partition of $V$, then $\sum_{\cX\in\binom{\cZ}{r}}\alpha(\cX)\le 1$.
\end{lemma}
\begin{proof}
  It is enough to observe that
  \[
    \sum_{\cX\in\binom{\cZ}{r}}\alpha(\cX) = |V|^{-r} \cdot \sum_{1 \le i_1 < \dots <i_r \le \ell} (|Z_{i_1}| \cdot \ldots \cdot |Z_{i_r}|) \le |V|^{-r} \cdot (|Z_1|+|Z_2|+\dots+|Z_\ell|)^r\, .
  \]
\end{proof}
Also observe that if $\cP$ is an equipartition of $V$,
% and $|V|\ge 2|\cP|$,
then for every $\cX\in\binom{\cP}{r}$,
\begin{equation}\label{eq:equi:alpha}
  \alpha(\cX) = |V|^{-r}\cdot \pi(\cX)
  \ge \left(\frac{1}{|V|}\cdot\floor{\frac{|V|}{|\cP|}}\right)^r
  \ge (2|\cP|)^{-r}.
\end{equation}

Given an $n$-vertex $r$-graph $H$ on $V$ and a family $\cX:=\{X_1,\dots,X_r\}$ of $r$ pairwise disjoint subsets of $V$, we define $e_H(\cX):=|E(X_1,\dots,X_r)|$,
\[
    d^p_H(\cX):=\frac{e_H(\cX)}{p \cdot \pi(\cX)} \quad \text{and} \quad \enr_H^p(\cX) := \alpha(\cX)\cdot\left(d_H^p(\cX)\right)^2\, .
\]
Moreover, for a family $\cP$ of at least $r$ pairwise disjoint subsets of $V$, we let
\[
  \enr_H^p(\cP) := \sum_{\cX\in\binom{\cP}{r}} \enr_H^p(\cX).
\]
The quantity $\enr_H^p(\cP)$ is often called the \defn{energy} of $\cP$.
We omit the subscript $H$ when it is clear from the context.
Moreover, we shall fix $p$ from now on and often omit it from the notation, understanding $d(\mathcal{X}), \mathcal{E}(\mathcal{X})$ to mean $d^p(\mathcal{X}), \mathcal{E}^p(\mathcal{X})$, etc. 
\medskip

\noindent \textbf{General properties of $\enr_H^p(\cX)$.}
We start by understanding how the function $\enr_H^p(\cX)$ changes if we refine the sets in $\cX$.
\newcommand{\vect}[1]{\mathbf{#1}}
\begin{lemma}\label{lem:dev}
  Let $H$ be an $r$-graph.
  Let $\cX=\{X_1,\dots,X_r\}$ be a family of pairwise disjoint subsets of $V$.
  Assume that for every $i\in[r]$,
  $X_i$ is partitioned into $\ell_i$ subsets $\{X_i^j\}_{j=1}^{\ell_i}$.
  For every $\vect{j}=(j_1,\dots,j_r)\in\prod_i[\ell_i]$,
  let $\cX_{\vect{j}}:=\{X_1^{j_1},\dots,X_r^{j_r}\}$.
  Let $\beta_{\vect{j}}:=\pi(\cX_{\vect{j}})/\pi(\cX)$
  and $\eps_{\vect{j}}:=d(\cX_{\vect{j}})-d(\cX)$.\footnote{We remark that the $\eps_{\vect{j}}$ can take both positive and negative values.}
  Then, we have
  \[
    \sum_{\vect{j}} \enr(\cX_{\vect{j}})
    = \enr(\cX) + \alpha(\cX) \sum_{\vect{j}}\beta_{\vect{j}}\eps_{\vect{j}}^2\, .
  \]
  In particular, 
  $\sum_{\vect{j}} \enr(\cX_{\vect{j}}) \geq \enr(\cX)$.
\end{lemma}

\begin{proof}
  Note that for every $\vect{j}$, $\alpha(\cX_{\vect{j}})=\alpha(\cX)\cdot\beta_{\vect{j}}$,
  and that $\sum_{\vect{j}}\beta_{\vect{j}}=1$.
  Note also that
  \[\begin{aligned}
    d(\cX) &= \frac{e(\cX)}{p \cdot \pi(\cX)}
           = \sum_{\vect{j}} \frac{e(\cX_{\vect{j}})}{p \cdot \pi(\cX)}
           = \sum_{\vect{j}} \beta_{\vect{j}}\cdot\frac{e(\cX_{\vect{j}})}{p \cdot \pi(\cX_{\vect{j}})}
           = \sum_{\vect{j}} \beta_{\vect{j}} \cdot d(\cX_{\vect{j}})\\
           &= \sum_{\vect{j}} \beta_{\vect{j}}(d(\cX)+\eps_{\vect{j}})
           = d(\cX) + \sum_{\vect{j}} \beta_{\vect{j}}\eps_{\vect{j}},
  \end{aligned}\]
  hence $\sum_{\vect{j}}\beta_{\vect{j}}\eps_{\vect{j}}=0$.
  Thus,
  \[\begin{aligned}
    \sum_{\vect{j}} \enr(\cX_{\vect{j}})
    &= \sum_{\vect{j}} \alpha(\cX_{\vect{j}}) \cdot (d(\cX_{\vect{j}}))^2
    = \alpha(\cX) \sum_{\vect{j}} \beta_{\vect{j}} (d(\cX)+\eps_{\vect{j}})^2\\
    &= \alpha(\cX) \sum_{\vect{j}} \beta_{\vect{j}}
        ((d(\cX))^2+2d(\cX)\eps_{\vect{j}}+\eps^2_{\vect{j}})\\
    &= \alpha(\cX)(d(\cX))^2\sum_{\vect{j}}\beta_{\vect{j}}
       +2\alpha(\cX)d(\cX)\sum_{\vect{j}}\beta_{\vect{j}}\eps_{\vect{j}} 
       +\alpha(\cX)\sum_{\vect{j}}\beta_{\vect{j}}\eps_{\vect{j}}^2\\
    &= \enr(\cX) + 0
       +\alpha(\cX)\sum_{\vect{j}}\beta_{\vect{j}}\eps_{\vect{j}}^2,
  \end{aligned}\]
  as required.
\end{proof}
\noindent
The previous claim gives the following immediate corollary.

\begin{cor}\label{cor:improv}
  Let $H$ be an $r$-graph.
  Let $\cX=\{X_1,\dots,X_r\}$ be a family of pairwise disjoint subsets of $V$.
  Suppose $\cX$ is not $(\eps,p)$-regular;
  namely, there exists $\cX_{\vect{1}}=\{X_1^1,\dots,X_r^1\}$
  with $X_i^1\subseteq X_i$ and $|X_i^1|\ge \eps|X_i|$
  for every $i=1,\dots,r$,
  for which $|d^p(\cX_{\vect{1}})-d^p(\cX)|>\eps$.
  For $i \in [r]$, set $X_i^2:=X_i\sm X_i^1$,
  and for every $\vect{j}=(j_1,\dots,j_r)\in [2]^r$,
  set $\cX_{\vect{j}}:=\{X_1^{j_1},\dots,X_r^{j_r}\}$.
  Then
  \[
    \sum_{\vect{j}\in[2]^r} \enr(\cX_{\vect{j}})
    \ge \enr(\cX) + \alpha(\cX)\cdot\eps^{r+2}.
  \]
\end{cor}

\begin{proof}
  This follows from Lemma~\ref{lem:dev}
  using the fact that $\alpha(\cX)$, $\beta_{\vect{j}}$ and $\eps_{\vect{j}}^2$ are all nonnegative,
  and noting that for $\vect{1}:=(1,\dots,1)$, it holds that
  $\beta_{\vect{1}}\ge\eps^r$
  and
  $\eps_{\vect{1}}^2>\eps^2$.
\end{proof}
\medskip

\noindent \textbf{Properties of $\enr^p_H(\cP)$ for upper-uniform hypergraphs.}
We assume from now on that $H$ is upper-uniform.
Firstly, we prove that then $\enr^p_H(\cP)$ is bounded.

\begin{lemma}\label{lem:bdd}
  If $H$ is an $n$-vertex $(\eta,p,D)$-upper-uniform $r$-graph
  and $\cP$ is a family of pairwise disjoint subsets of $V$ such that each set has size at least $\eta n$,
  then $\enr(\cP)\le D^2$.
\end{lemma}

\begin{proof}
  The lemma follows from upper-uniformity and Lemma~\ref{lem:alpha} as 
  \[
    \enr(\cP) = \sum_{\cX\in\binom{\cP}{r}}\alpha(\cX) \cdot (d(\cX))^2 \le D^2 \cdot \sum_{\cX\in\binom{\cP}{r}}\alpha(\cX) \le D^2\, . \qedhere
  \]
\end{proof}

Secondly we prove that, as long as $H$ is an upper-uniform $r$-graph, for a family $\cX:=\{X_1,\dots,X_r\}$ of pairwise disjoint subsets of vertices, if $X_i^*$ is a large subset of $X_i$ for each $i \in [r]$ and $\cX^*:=\{X_1^*, \dots, X_r^*\}$, then $d(\cX)$ and $d(\cX^*)$ do not differ much.

\begin{lemma}\label{lem:lose}
  Let $0 \le \delta \le 1/2$ and $H$ be an $n$-vertex $(\eta, p, D)$-upper-uniform $r$-graph.
  Let $\cX=\{X_1,\dots,X_r\}$ be a family of pairwise disjoint subsets of $V$ with $\delta|X_i| \ge \eta n$ for each $i \in [r]$.
  Let $X_i^* \subseteq X_i$ with $|X_i^*| \ge (1-\delta)|X_i|$ for each $i \in [r]$, and define $\cX^*:=\{X_1^*, \dots, X_r^*\}$.
  Then $|d(\cX) - d(\cX^*)| \le D \delta r$ and 
  $|(d(\cX))^2-(d(\cX^*))^2| \le 2D^2\delta r$.
\end{lemma}

\begin{proof}
  Observe that
  \[
    (1-\delta)^r d(\cX^*) = \frac{1}{p} \cdot \frac{(1-\delta)^r}{\pi(\cX^*)} \cdot e(\cX^*) \le \frac{1}{p} \cdot \frac{1}{\pi(\cX)} \cdot e(\cX^*) \le \frac{e(\cX)}{p \cdot \pi(\cX)} = d(\cX)\, .
  \]
  Using $(1-\delta)^r \ge 1-\delta r$ and the inequality above, we get that $d(\cX^*) - \delta r d(\cX^*) \le d(\cX)$, which can be rearranged as 
  \begin{equation}
  \label{eq:reg_eq_1}
    d(\cX^*) \le d(\cX) + \delta r d(\cX^*) \le d(\cX) + D \delta r\, ,
  \end{equation}
  where the last inequality uses that $H$ is upper-uniform.

  For each $i \in [r]$, take $\tilde{X_i} \subseteq X_i^*$ with $|\tilde{X_i}| = (1-\delta)|X_i|$, and define $\cZ_i:=\cX \cup \{X_i \setminus \tilde{X_i}\} \setminus \{X_i\}$.
  Then $e(\cX) \le e(\cX^*) + \sum_{i \in [r]} e(\cZ_i)$ and thus
  \begin{align*}
    d(\cX) &\le \frac{e(\cX^*)}{p \cdot \pi(\cX)} + \sum_{i \in [r]} \frac{e(\cZ_i)}{p \cdot \pi(\cX)} \\
    &\le \frac{e(\cX^*)}{p \cdot \pi(\cX^*)} + D \cdot \sum_{i \in [r]} \frac{|X_i \setminus \tilde{X_i}|}{|X_i|} = d(\cX^*) + D \delta r\, ,
  \end{align*}
  where the second inequality uses upper-uniformity.
  Together with~\eqref{eq:reg_eq_1}, we get $|d(\cX) - d(\cX^*)| \le D \delta r$.
  Therefore $|(d(\cX))^2 - (d(\cX^*))^2| = |d(\cX) - d(\cX^*)| \cdot |d(\cX) + d(\cX^*)| \le (D \delta r) \cdot (2D) \le 2 D^2 \delta r$, where $|d(\cX) + d(\cX^*)| \le 2D$ follows again from upper-uniformity. 
  This finishes the proof.
\end{proof}
\medskip

\noindent \textbf{Key lemma and proof of the regularity lemma.}
We are now ready to state and prove the key lemma in the proof of the regularity lemma.
As we are interested in a multicolour version of the regularity lemma, we need to extend the definition of $\enr^p$ as follows.
For a collection $\cH:=\{H_1,\dots,H_q\}$ of $q$ hypergraphs on the same vertex set $V$ and a family $\cP$ of pairwise disjoint subsets of $V$, we define
\[
  \enr^p_\cH(\cP) := \sum_{j=1}^q\enr^p_{H_j}(\cP).
\]

\begin{lemma}[Key lemma]\label{lem:key}
  Let $\eta  \ll 1/t \ll \eps, 1/D, 1/q, 1/r$ with $D,q,t,r \in \mathbb{N}$ and $D > 1$. 
  Then for every $p \in [0, 1]$ the following holds.  
  Let $H_1, \dots, H_q$ be $(\eta, p, D)$-upper-uniform $n$-vertex $r$-graphs on vertex set $V$. 
  Suppose that $\cP=\{V_1,\dots,V_t\}$
  is an equipartition of $V$ with $|V_i| \ge 8^{t^{r-1}}$ for every $i \in [t]$, 
  that is not $(\eps,p)$-regular
  with respect to $H_j$ for some $j \in [q]$.
  Then, there exists an equipartition $\cP'$ of $V$ with $t(4^{t^{r-1}}-2^{t^{r-1}})$ parts satisfying
  $\enr^p_\cH(\cP')\ge \enr^p_\cH(\cP)+\frac{\eps^{r+3}}{r^r \cdot 2^{2r+4}}$.
\end{lemma}

\begin{proof}
  Let $m:=\lfloor n/t \rfloor$ and note that $|V_i| \in \{m,m+1\}$ for each $i \in [t]$.  
  For $I=\{i_1,\dots,i_r\} \in \binom{[t]}{r}$, denote $\cP_{I}:=\{V_{i_1},\dots,V_{i_r}\}$.
  Assume without loss of generality that $\cP$ is not $(\eps,p)$-regular
  with respect to $H_1$.
  Let $\cI\subseteq\binom{[t]}{r}$ be the set of $r$-subsets $I$ of $[t]$
  for which $\cP_{I}$ is not $(\eps,p)$-regular with respect to $H_1$.
  Thus, $|\cI|\ge\eps \binom{t}{r} \ge \eps \cdot (t/r)^r$.
  Consider any $I\in \cI$.
  Then 
  there exist $V_i^{I,1}\subseteq V_i$, $i \in I$,
  such that $|V_i^{I,1}|\ge\eps|V_i|$
  and $|d_{H_1}(\cP_{I})-d_{H_1}(\cP_{I}^{\vect{1}})|>\eps$,
  where $\cP_{I}^{\vect{1}}=\{V_i^{I,1}\}_{i\in I}$.
  Let us also set $V_i^{I,2}:=V_i\sm V_i^{I,1}$.
  
  Fix any $i \in [t]$ and consider the atoms
  of the Venn diagram of the sets
  $$\left\{V_i^{I,x}\ :\ i\in I\in \cI,\ x\in\{1,2\}\right\}\, .$$
  Since $i$ is contained in at most $\binom{t-1}{r-1}$ sets $I \in \cI$,
  we deduce that $V_i$ is refined into at most $2^{\binom{t-1}{r-1}} \le 2^{t^{r-1}}$ atoms.
  Set $\ell:=4^{t^{r-1}}-2^{t^{r-1}}$ and $d=\lfloor m/4^{t^{r-1}} \rfloor$.
  Observe that $d+1 \ge 2^{t^{r-1}}$ (by the assumption on the size of $V_i$) and $|V_i| \le m+1 \le 4^{t^{r-1}}(d+2)$.
  Let $W_{ih} \subseteq V_i$, $1 \le h \le \ell$, be pairwise disjoint sets of size $d$ such that each $W_{ih}$ is contained in some atom. This is possible since all but at most $d-1$ elements of each atom can be partitioned into sets of size $d$, and
  \[
    \ell d+2^{t^{r-1}}(d-1) \le m \le |V_i|\, .
  \]  
  Set $\tilde{V}_i:=\cup_{h=1}^{\ell}W_{ih}$. Also, for each $I \in \cI$ with $i \in I$ and for each $x \in \{1,2\}$, set 
  $$
  \tilde{V}_i^{I,x}:=\bigcup \{W_{ih}: W_{ih} \subseteq V_i^{I,x}\}\, .
  $$
  Finally, set $\tilde{\cP}:=\{\tilde{V}_i: i \in [t]\}$ and $\cP':=\{W_{ih}: i \in [t], h \in [\ell]\}$.
  We will compare the energies of $\cP$ and $\cP'$ by comparing both with $\tilde{\mathcal{P}}$.
  
  Setting $\delta:=\frac{\eps^{r+3}}{2^{4r+5}D^2 qr^{r+1}}$, we start by observing that
  \begin{equation}
  \label{eq:size_V_i}
    \frac{|V_i \setminus \tilde{V}_i|}{|V_i|} = 
    \frac{|V_i| - \ell d}{|V_i|} 
    \le 1 - \frac{\ell d}{4^{t^{r-1}} (d+2)} \le \frac{2}{d+2}+\frac{1}{2^{t^{r-1}}} \le \frac{3}{2^{t^{r-1}}} \le \delta \, ,
  \end{equation}
  where the last inequality holds as $t$ is large enough as a function of $\varepsilon,r,q,D$.
  So we see that $|\tilde{V}_i| \ge (1-\delta)|V_i|$. In addition,
  $\delta |V_i| \ge \eta n$, as $\eta$ is small enough as a function of $t$.
  Now, fix any $I \in \binom{[t]}{r}$, 
  and set $\tilde{\mathcal{P}}_I := \{\tilde{V}_i : i \in I\}$ (recall also that $\mathcal{P}_I = \{V_i : i \in I\}$).
  By Lemma~\ref{lem:lose}, for each $j \in [q]$ we have
  \[ 
    |d_{H_j}(\cP_I)-d_{H_j}(\tilde{\cP}_I)| \le D \delta r\, 
  \]
  and 
  \[ 
    |(d_{H_j}(\cP_I))^2-(d_{H_j}(\tilde{\cP}_I))^2| \le 2D^2\delta r.
  \]
  Moreover, using $|\tilde{V}_i| \ge (1-\delta)|V_i|$ again, we have $|V_i|/|\tilde{V}_i| \le \frac{1}{1-\delta} \leq 1+2\delta$.
  For $\cX:=\{V_{i_1},\dots,V_{i_r}\} \in \binom{\cP}{r}$, define $\tilde{\cX}:=\{\tilde{V}_{i_1},\dots,\tilde{V}_{i_r}\}$.
  It follows that $\alpha(\cX) \le (1+2\delta)^r \alpha(\tilde{\cX})$ and thus for each $j \in [q]$,
  \begin{align*}
    \enr_{H_j}(\cP) &= \sum_{\cX \in \binom{\cP}{r}} \alpha(\cX) \left(d_{H_j} (\cX)\right)^2 \le 2D^2 \delta r + \sum_{\cX \in \binom{\cP}{r}} \alpha(\cX) \left(d_{H_j}(\tilde{\cX})\right)^2 \\
    & \le 2D^2 \delta r+ (1+2\delta)^r \sum_{\cX \in \binom{\cP}{r}} \alpha(\tilde{\cX}) \left(d_{H_j}(\tilde{\cX})\right)^2 \\
    & =(1+2\delta)^r \enr_{H_j}(\tilde{\cP}) + 2D^2\delta r \le \enr_{H_j}(\tilde{\cP}) + 4^r D^2 \delta+ 2D^2\delta r\\
    &\le \enr_{H_j}(\tilde{\cP}) + 4^{r+1} \delta D^2 r\, ,
  \end{align*}
  where we used $(1+2\delta)^r \le 1+4^r \delta$ and $\enr_{H_j}(\tilde{\cP}) \le D^2$ (which holds by Lemma~\ref{lem:bdd}) for the penultimate inequality.
  Summing over all $j \in [q]$, we get
  \begin{equation}
  \label{eq:partition}
    \sum_{j=1}^q \enr_{H_j}(\cP) \le \sum_{j=1}^q \enr_{H_j}(\tilde{\cP}) + q \cdot 4^{r+1} \delta D^2 r = \enr(\tilde{\cP}) + r^{-r} \cdot 2^{-2r-3} \cdot \eps^{r+3}\, ,
  \end{equation}
  where the last equality uses the definition of $\delta$.

  Let $I \in \cI$ and $i \in I$. Then, using $|V_i^{I,1}| \ge \eps |V_i|$ and~\eqref{eq:size_V_i},
  \begin{equation}
  \label{eq:difference_in_size}
    \frac{|V_i^{I,1} \setminus \tilde{V}_i^{I,1}|}{|V_i^{I,1}|}
    \le \frac{|V_i \setminus \tilde{V}_i|}{|V_i^{I,1}|}
    = \frac{|V_i \setminus \tilde{V}_i|}{|V_i|} \cdot \frac{|V_i|}{|V_i^{I,1}|}
    \le \delta/\eps\, .
  \end{equation}
  It follows from Lemma~\ref{lem:lose} (with parameter $\delta/\eps$ in place of $\delta$) that for each $j \in [q]$,
  \[
    |d_{H_j}(\cP_I^\vect{1})-d_{H_j}(\tilde{\cP}_I^\vect{1})| \le D \delta r/\eps\, ,
  \]
  where 
  $\tilde{\mathcal{P}}_I^{\vect{1}} := 
  \{\tilde{V}_i^{I,1} : i \in I\}$.
  In particular, for $j=1$, and since $I \in \cI$,
  \begin{align}
  \label{eq:irregular}
    |d_{H_1}(\tilde{\cP}_I^\vect{1})-d_{H_1}(\tilde{\cP}_I)| &\ge |d_{H_1}(\cP_I^\vect{1})-d_{H_1}(\cP_I)| - |d_{H_1}(\tilde{\cP}_I^\vect{1})-d_{H_1}(\cP_I^\vect{1})| - |d_{H_1}(\tilde{\cP}_I)-d_{H_1}(\cP_I)| \nonumber \\
    & > \eps -  D \delta r/\eps - D \delta r > \eps - 2D \delta r/\eps > \eps/2\, ,
  \end{align}
  where the last inequality follows from the choice of $\delta$.

  Observe that $\cP'$ is a refinement of $\tilde{\cP}$.
  Moreover, for each $I \in \cI$, since $|V_i^{I,1}| \ge \eps |V_i|$ and by~\eqref{eq:difference_in_size}, we have $|\tilde{V}_i^{I,1}| \ge (1-\delta/\eps) |V_i^{I,1}| \ge \eps |V_i|/2$.
  Together with~\eqref{eq:irregular}, we conclude that $\tilde{\cP}_I$ is not an $(\varepsilon/2,p)$-regular tuple within the partition $\tilde{\cP}$ for the graph $H_1$, as witnessed by the sets $(\tilde{V}_i^{I,1})_{i \in I}$.
  Now, by Lemma~\ref{lem:dev}, Corollary~\ref{cor:improv} and~\eqref{eq:equi:alpha}, we have
  \begin{align*}
    \enr_{\mathcal{H}}(\cP') &= \enr_{H_1}(\cP') + \sum_{j=2}^q \enr_{H_j}(\cP') \ge \sum_{j=1}^q\enr_{H_j}(\tilde{\cP}) + \sum_{I \in \cI} \alpha(\tilde{\cP}_I) \cdot (\eps/2)^{r+2} \\
    & \ge \enr_{\mathcal{H}}(\tilde{\cP}) + \eps (t/r)^r \cdot (2t)^{-r} \cdot (\eps/2)^{r+2} =  \enr_{\cH}(\tilde{\cP}) + r^{-r} \cdot 2^{-2r-2} \cdot \eps^{r+3}\, .
  \end{align*}
  Together with~\eqref{eq:partition}, we conclude that
  \begin{align*}
    \enr_{\mathcal{H}}(\cP') & \ge 
    \enr_{\mathcal{H}}(\tilde{\cP}) + 
    \frac{\eps^{r+3}}{r^r \cdot 2^{2r+2}}  \ge 
    \enr_{\mathcal{H}}(\cP) - \frac{\eps^{r+3}}{r^r \cdot 2^{2r+3}}+ \frac{\eps^{r+3}}{r^r \cdot 2^{2r+2}} =  \enr_{\mathcal{H}}(\cP) + \frac{\eps^{r+3}}{r^r \cdot 2^{2r+3}} \, .
  \end{align*}

  The collection $\cP'$ has order $t\ell$ and each set has size $d$.
  Moreover, there are at most $t \cdot 2^{\binom{t-1}{r-1}} \cdot d$ vertices of $V$ not covered by $\cP'$.
  We let $\cP''$ be the equipartition of $V$ obtained from $\cP'$ by distributing all these vertices among the parts of $\cP'$ as evenly as possible.
  Note that $\cP''$ still has $t\ell$ parts. 
  We are left to bound $\enr_{\cH}(\cP'')$.

  Observe that each set of $\cP'$ gets at most 
  $\frac{t \cdot 2^{\binom{t-1}{r-1}} \cdot d}{t \cdot \ell} \leq 
  \frac{2^{t^{r-1}}}{4^{t^{r-1}} - 2^{t^{r-1}}} \cdot d \le \delta d$ new vertices, where we used the choice of $\ell$ and that $t$ is sufficiently large.
  Then if $C'_1, \dots, C'_r$ are sets in $\cP'$ and $C''_1, \dots, C''_r$ are the corresponding sets in $\cP''$, we have for each $j=1, \dots, q$
  \begin{align*}
    d_{H_j}(C''_1,\dots,C''_r) = \frac{e_{H_j}(C''_1,\dots,C''_r)}{|C''_1| \cdot \dots \cdot |C''_r|} &\ge \frac{e_{H_j}(C'_1,\dots,C'_r)}{|C'_1| \cdot \dots \cdot |C'_r|} \cdot \frac{|C'_1| \cdot \dots \cdot |C'_r|}{|C''_1| \cdot \dots \cdot |C''_r|} \\
    & \ge d_{H_j}(C'_1,\dots,C'_r) \cdot (1+\delta)^{-r}\, .
  \end{align*}
  Therefore
  $$
    \enr_{\cH}(\cP'') \ge (1+\delta)^{-2r} \cdot \enr_{\cH}(\cP') \ge \enr_{\cH}(\cP') - 2r\delta q D^2\, ,
  $$
  where we used $(1+\delta)^{-2r} \ge (1-\delta)^{2r} \ge 1-2r\delta$ and $\enr_{H_j}(\tilde{\cP'}) \le D^2$ for each $j \in [q]$ (which holds by Lemma~\ref{lem:bdd}).
  We conclude that
  \[
    \enr_{\cH}(\cP'') \ge \enr_{\cH}(\cP') - 2r\delta q D^2 \ge \enr_{\cH}(\cP) + \frac{\eps^{r+3}}{r^r \cdot 2^{2r+3}} - 2r\delta q D^2 \ge \enr_{\cH}(\cP) + \frac{\eps^{r+3}}{r^r \cdot 2^{2r+4}} \, ,
  \]
  where the last inequality holds by the choice of $\delta$. Thus, $\cP''$ is the desired equipartition.
\end{proof}

\noindent
Finally, we can complete the proof of the multicolour weak sparse hypergraph regularity lemma.

\lateproof{Lemma~\ref{lem:reg:clr}} 
    Let $t_0$ be large enough as a function of $\eps, D,q,t,r$ so that Lemma~\ref{lem:key} is applicable for every $t \geq t_0$.
    For $i \ge 0$, define $t_{i+1}:=t_i(4^{t_i^{r-1}}-2^{t_i^{r-1}})$.
    Finally set $s:=\frac{r^r \cdot 2^{2r+4}}{\eps^{r+3}}\cdot qD^2$, $T:=t_s \cdot 8^{t_s^r-1}$ and let $\eta$ be small enough for Lemma~\ref{lem:key} to be applicable for every $t \leq t_s$.

    We can assume $n \ge T$ as, otherwise, the trivial partition into singletons satisfies the lemma.
    We start from an arbitrary equipartition $\cP_0$ of $V$ of order $t_0$, and repeatedly apply Lemma~\ref{lem:key} until we get an equipartition which is $(\eps,p)$-regular with respect to $H_j$ for every $j \in [q]$. 
    We will briefly show that the number of steps in the process is at most $s$, and thus, by the guarantees of Lemma~\ref{lem:key}, the size of the final partition is at most $t_s \le T$. As already mention, the condition in Lemma~\ref{lem:key} that $\eta$ is small enough in terms of $t$ holds in all of our applications of Lemma~\ref{lem:key}. Also, each part of each partition has size at least $\lfloor n/t_s \rfloor \geq 8^{t_s^{r-1}}$, given our assumption on $n$. Thus, we may indeed apply Lemma~\ref{lem:key}. 
    
    By Lemma~\ref{lem:bdd} and since $\eta \ll 1/T$, we have that all the equipartitions $\cP$ obtained above satisfy $0 \le \enr_{\cH}(\cP) \le qD^2$. Since each application of Lemma~\ref{lem:key} increases $\enr_{\cH}(\cdot)$ by at least $\frac{\eps^{r+3}}{r^r \cdot 2^{2r+4}}$, this process must stop after at most $s$ iterations, as claimed.
\endproof

\end{document}